\documentclass[reqno]{amsart}
\usepackage{amsmath,amssymb,verbatim,geometry,color,esint}
\usepackage{graphicx} 
\usepackage{tikz}
\usepackage{tikz-3dplot}
\usepackage[capitalise]{cleveref}
\newcommand{\R}{\mathbb{R}}
\newcommand{\Z}{\mathbb{Z}}

\newcommand{\cM}{\mathcal{M}}

\newcommand{\cP}{\mathcal{P}}
\newcommand{\cQ}{\mathcal{Q}}

\DeclareMathOperator{\reg}{reg}
\DeclareMathOperator{\sing}{sing}
\DeclareMathOperator{\sym}{sym}

\DeclareMathOperator{\Star}{Star}

\newcommand{\simto}{\overset\sim\to}

\newcommand{\qand}{{\quad\text{and}\quad}}

\numberwithin{equation}{section}       

\newtheorem{prop} {Proposition} [section]
\newtheorem{thm}[prop] {Theorem} 
\newtheorem*{thmA}{Theorem A} 
\newtheorem*{thmB}{Theorem B} 
\newtheorem*{thmC}{Theorem C} 

\newtheorem{lem}[prop] {Lemma}

\newtheorem{rmk}[prop] {Remark}

\newtheorem{exam}[prop] {Example}

\title[Monge--Amp\`ere equations and regularity]{Regularity of the solution to a real Monge--Amp\`ere equation on the boundary of a simplex}
\date{\today}

\author[AHJMM]{Rolf Andreasson \and Jakob Hultgren \and Mattias Jonsson \and Enrica Mazzon \and Nicholas McCleerey}

\address{Dept of Mathematical Sciences\\
Chalmers University of Technology\\
412 96 G\"oteborg\\
Sweden}

\email{rolfan@chalmers.se}

\address{Dept of Mathematics and Mathematical Statistics\\
  Ume{\aa} University\\ 
  901 87 Ume{\aa}\\
  Sweden}

\email{jakob.hultgren@umu.se}
  
\address{Dept of Mathematics\\
  University of Michigan\\
  Ann Arbor, MI 48109-1043\\ USA}

\email{mattiasj@umich.edu}

\address{Fakult\"at für Mathematik\\
Universit\"at Regensburg\\
93040 Regensburg\\
Germany}

\email{e.mazzon15@alumni.imperial.ac.uk}

\address{Dept of Mathematics\\
Purdue University\\
West Lafayette, Indiana 47907-2067\\
USA}  

\email{nmccleer@purdue.edu}

\begin{document}

\begin{abstract}
    Motivated by conjectures in Mirror Symmetry, we continue the study of the real Monge--Amp\`ere operator on the boundary of a simplex. This can be formulated in terms of optimal transport, and we consider, more generally, the problem of optimal transport between symmetric probability measures on the boundary of a simplex and of the dual simplex. For suitably regular measures, we obtain regularity properties of the transport map, and of its convex potential. To do so, we exploit boundary regularity results for optimal transport maps by Caffarelli, together with the symmetries of the simplex.
\end{abstract}
\maketitle

\setcounter{tocdepth}{1}
\tableofcontents

\section*{Introduction}
A celebrated result by Yau~\cite{Yau}, of fundamental importance to complex geometry, shows that suitable complex Monge--Amp\`ere equations on compact K\"ahler manifolds admit smooth solutions. 

A corresponding result for the \emph{real} Monge--Amp\`ere equation was proved by Cheng and Yau in~\cite{ChengYau} and takes place on a compact special affine manifold, i.e.\ a compact manifold equipped with a volume form, an atlas for which the transition functions are affine and volume preserving; see also~\cite{HO19,GuedjTo}. Their result has important consequences: combined with the J\"orgens--Calabi--Pogorelov theorem, it completely characterizes compact special affine manifolds admitting Hessian metrics as quotients of real tori. However, for applications related to Lagrangian fibrations, the assumption of smoothness is often too strong. In particular, 
developments in 
mirror symmetry---notably the SYZ conjecture---motivate a study of the real Monge--Amp\`ere equation on compact manifolds equipped with a \emph{singular} special affine structure, i.e.\ a special affine structure outside a subset of codimension two. It is nontrivial to even make sense of the Monge--Amp\`ere equation in the presence of singularities, but breakthrough work by Yang Li~\cite{LiFermat,LiSYZ} on the SYZ conjecture has sparked progress in this area~\cite{HJMM24,AH23,LiFano}. Here we follow up on~\cite{HJMM24}, where the manifold is the boundary of a simplex. 

Fix $d\in\Z_{\ge1}$, and let $\Delta$ be a $(d+1)$-dimensional simplex, embedded in a $(d+1)$-dimensional vector space $V$ such that the barycenter of $\Delta$ is the origin. We can equip $A:=\partial\Delta$ with a singular special affine structure, in which the volume form is Lebesgue measure, normalized to mass one, and the regular set $A_{\reg}$ is the union of the interiors of the $d$-dimensional faces together with the open stars of the vertices of $A$ in the barycentric subdivision of $A$. The co-dimension of $A_{\reg}$ is 2. See Section~\ref{sec:Setup} for details.

Denote by $\Delta^\vee\subset V^\vee$ the polar simplex, consisting of linear functions $\ell\colon V\to\R$ with $\ell|_\Delta\le1$, and let $\cP\subset C^0(A)$ be the closed convex set consisting of restrictions to $A$ of convex functions $\phi\colon V\to\R$ such that $\phi=\sup_{\ell\in\Delta^\vee}\ell+O(1)$. We can equip $A$ with a principal $\R$-bundle $\Lambda_A$, affine over $A_{\reg}$, and any $\phi\in\cP$ induces a continuous metric on $\Lambda_A$ that is convex over $A_{\reg}$ (see \cite[\S3.4]{HJMM24}).

In~\cite{HJMM24}, the last four authors constructed a Monge--Amp\`ere operator that takes symmetric functions in $\cP$ to symmetric probability measures on $A$. More precisely, there is a unique action of the permutation group $G=S_{d+2}$ on $V$ by linear transformations preserving $\Delta$ and permuting the vertices. Let $\cP_{\sym}\subset\cP$ be the set of $G$-invariant functions. Then it follows from~\cite[Theorem~B]{HJMM24} that there exists a unique continuous map $\phi\mapsto\mu_\phi$ from $\cP_{\sym}$ to the space $\cM_{\sym}(A)$ of symmetric probability measures on $A$ such that $\mu_\phi|_{A_{\reg}}$ is the Monge--Amp\`ere measure of the convex metric induced by $\phi$.
Moreover, this operator induces a homeomorphism $\cP_{\sym}/\R\simto\cM_{\sym}(A)$. 

\medskip
A key motivation for studying the Monge--Amp\`ere equation on the boundary on a simplex (or more general reflexive polytopes) comes from applications to Mirror Symmetry and specifically the SYZ conjecture: see~\cite{SYZ,GS06,Gro13,KS06, Loftin, MP21,PS22}. However, here we will consider the Monge--Amp\`ere equation in its own right. Specifically, in  analogy to Yau's theorem, we shall study the regularity of the solution $\phi\in\cP_{\sym}$ to the equation $\mu_\phi=\mu$, for suitably regular measures $\mu\in\cM_{\sym}(A)$. 
Our first theorem reads as follows.
\begin{thmA}\label{thm:RegularityOfMetric}
    Let $\Delta$ be a simplex, and $\mu$ a symmetric probability measure on $A=\partial\Delta$ of the form $\mu=f\,d\mu_A$, where $\mu_A$ is Lebesgue measure on $A$ and $f\in C^\infty(A_{\reg})$ satisfies $0<\inf f\le\sup f<\infty$. Then the metric on $A_{\reg}$ induced by $\phi$ is $C^\infty$, strictly convex, and uniformly $C^{1,\alpha}$ for some $\alpha>0$. 
\end{thmA}
This effectively produces a Monge--Amp\`ere metric outside a set of codimension 2, confirming a widespread expectation related to the SYZ-conjecture. The $C^\infty$ regularity and strict convexity was proved in the work~\cite{AH23} by the first two authors. Note that this goes well beyond local regularity theory for real Monge--Amp\`ere equations, due to well-known counterexamples by Pogorelov (see, for example, \cite{Moo15}).
Using the symmetry in the situation,~\cite{AH23} reduces the $C^\infty$ property
and strict convexity to the inner interior regularity result of Caffarelli~\cite{Caf92a}.
Here we will derive the uniform $C^{1,\alpha}$-regularity (in each of the finitely many charts used to cover $A_{\reg}$) 
in the same way, but using the boundary regularity results of Caffarelli~\cite{Caf92b}. At least in dimension $d=2$, the solution is not $C^{1,1}$, as shown in~\cite{JMPS23}. Note that in the case of compact Hessian manifolds, regularity of weak solutions follows by Caffarelli and Viaclovsky \cite{CV}. We note that the arguments presented for the simplex here also hold for any Weyl polytope, as detailed in \cite{DH} (which was written after this appeared as preprint). 

To explain how Caffarelli's results apply, we need to be more precise about the definition of the Monge--Amp\`ere operator $\phi\mapsto\mu_\phi$. 
Write $B:=\partial\Delta^\vee\subset V^\vee$. We define $\cQ_{\sym}\subset\cQ\subset C^0(B)$ in the same way as $\cP_{\sym}\subset\cP\subset C^0(A)$. The \emph{c-transform} of a bounded function $\phi\colon A\to\R$ is the function $\phi^c\colon B\to\R$ defined by 
\[
\phi^c(y)=\sup_{x\in A}\langle x,y\rangle - \phi(x),
\]
where $\langle\cdot,\cdot\rangle$ denotes the canonical pairing between $V$ and $V^\vee$. The c-transform $\psi^c\colon A\to\R$ of a bounded function $\psi\colon B\to\R$ is defined in the same way. We always have $\phi^{cc}\le\phi$,  
and it turns out that $\phi\in\cP$ iff $\phi$ is \emph{c-convex} in the sense that $\phi^{cc}=\phi$ (resp.\ we have $\psi^{cc}\le\psi$, and $\psi\in\cQ$ iff $\psi^{cc}=\psi$). Moreover, the c-transforms restrict to isometries $\cP\simeq\cQ$ and $\cP_{\sym}\simeq\cQ_{\sym}$. If $\phi\in\cP$ and $x\in A$, then the \emph{c-subgradient} $(\partial^c\phi)(x)$ is the nonempty set of $y\in B$ such that $\phi(x)+\phi^c(y)=\langle x,y\rangle$. When the c-subgradient is a singleton, we instead say \emph{c-gradient}. 
It is proved in~\cite{HJMM24} that if $\phi\in\cP_{\sym}$, then $\partial^c\phi\colon A\to B$ is single-valued Lebesgue a.e.\ and Lebesgue measurable. Any $\psi\in\cQ_{\sym}$ similarly induces $\partial^c\psi\colon B\to A$. The Monge--Amp\`ere operators $\cP_{\sym}\to\cM_{\sym}(A)$ and $\cQ_{\sym}\to\cM_{\sym}(B)$ are now defined by 
\[
\mu_\phi:=(\partial^c\phi)_*\mu_B,
\quad\text{and}\quad
\mu_\psi:=(\partial^c\psi)_*\mu_A,
\]
where $\mu_A$ and $\mu_B$ are the symmetric Lebesgue measures on $A$ and $B$, respectively, normalized to mass one. 
It turns out 
that in suitable coordinates, $\phi$ is the convex potential for the optimal transport map of Lebesgue like measures on convex domains, and thus the regularity results of Caffarelli apply.

\smallskip
In fact, we prove Theorem~A in a more general context. Let $\mu$ and $\nu$ be symmetric probability measures on $A$ and $B$ that are equivalent to Lebesgue measure. Adapting arguments from~\cite{HJMM24}, we show that there exists $\phi\in\cP_{\sym}$, unique up to an additive constant, such that
\begin{equation}\tag{$\star$}
    (\partial^c\phi)_*\mu=\nu
    \quad\text{and}\quad
    (\partial^c\psi)_*\nu=\mu,
\end{equation}
where $\psi=\phi^c$, see Theorem~\ref{thm:transport}. Using Caffarelli's results, we show that if the densities of $\mu$ and $\nu$ are smooth and bounded away from zero and infinity, then $\phi$ and $\psi$ define metrics on $\Lambda_A$ and $\Lambda_B$, respectively, with the same regularity properties as in Theorem~A. In fact, by \cite{JS22}, the $C^{1,\alpha}$ property holds whenever the measures are absolutely continuous and doubling. See Theorem~\ref{thm:breg} for a precise statement and a definition of doubling in this setting. 

\smallskip
We now want to investigate the regularity of $\phi$ and $\psi$ on the singular locus. 

\begin{thmB} \label{thm:cGrad}
Let $\mu$ and $\nu$ be symmetric probability measures on $A$ and $B$ that are absolutely continuous with respect to Lebesgue measure, with densities that are bounded away from zero and infinity (or more generally, densities that are doubling). Assume that $\phi\in\cP_{\sym}$ and $\psi=\phi^c\in\cQ_{\sym}$ satisfy~($\star$). Then the c-gradients of $\phi$ and $\psi:=\phi^c$ exist (i.e.\ are single-valued) at every point of $A$ and $B$, respectively. Moreover, 
\[
 \partial^c\phi\colon A\to B
 \quad\text{and}\quad
 \partial^c\psi\colon B\to A
\]
are H\"older continuous homeomorphisms in the Euclidean metrics that are inverse one to another and map $A_{\reg}$ onto $B_{\reg}$.
\end{thmB}

The fact that the c-gradients exist and are continuous can be seen as a differentiability property which extends to the singular locus, even though there is no obvious smooth structure there. In this sense, it rules out a type of $Y$-shaped singularities studied in \cite{Moo21,MR23}.

That the homeomorphisms $\partial^c\phi$ and $\partial^c\psi$ preserve the regular and singular loci is a consequence of them being equivariant with respect to the action of the symmetry group $G\simeq S_{d+2}$. This equivariance gives more information. For example, in dimension three, the singular loci of $A$ and $B$ consists of 30 line segments each, and these are matched up by $\partial^c\phi$ and $\partial^c\psi$. The endpoints of these intervals come in two flavors, known as positive and negative vertices, and $\partial^c\phi$ and $\partial^c\psi$ take positive vertices to negative vertices, and conversely. See Figure~\ref{fig:signs} and~\cref{rmk:3dvert1} for more information.

The terminology `positive' and `negative' follows~\cite{Joyce}, and is based on the monodromies at a vertex. The vertices represent the most singular part of the conjectural global SYZ fibration on the Fermat family near the large complex structure limit. It has been advocated, in~\cite{LiThreefolds} for example, that to obtain a global version of the SYZ conjecture for threefolds, one should study the behavior of the solution of the real Monge-Amp\`ere equation near the singular set, in particular near the vertices. See also~\cref{rmk:3dvert2}.
\begin{figure}[h]
\begin{tikzpicture}[scale=0.75]
	\coordinate (2) at (0,0);
	\coordinate (4) at (0,3);
	\coordinate (3) at (3,0);
	\coordinate (1) at (-3,-1);
	\coordinate (5) at (2,-2);
	\coordinate (12) at (-1.5,-0.5);
	\coordinate (14) at (-1.5,1);
	\coordinate (15) at (-0.5,-1.5);
	\coordinate (13) at (0,-0.5);
	\coordinate (23) at (1.5,0);
	\coordinate (24) at (0,1.5);
	\coordinate (25) at (1,-1);
	\coordinate (34) at (1.5,1.5);
	\coordinate (35) at (2.5,-1);
	\coordinate (45) at (1,0.5);
	\coordinate (124) at (-1,2/3);
	\coordinate (145) at (-1/3,0);
	\coordinate (125) at (-1/3,-1);
	\coordinate (123) at (0,-1/3);
	\coordinate (135) at (4/3,-1);
	\coordinate (234) at (1,1);
	\coordinate (245) at (2/3,1/3);
	\filldraw (3) circle (1pt);
	\filldraw (4) circle (1pt);
	\filldraw (5) circle (1pt);
	\filldraw (1) circle (1pt);
	\node[right] at (3) {$m_3$};
	\node[above] at (4) {$m_4$};
	\node[below] at (5) {$m_0$};
	\node[left] at (1) {$m_1$};
	\draw[dotted] (2)--(3)--(4)--(2)--(1)--(5)--(3)--(1)--(4)--(5)--(2);
	\draw[thick] (14)--(145)--(15);
	\draw[thick] (45)--(145);
	\node[above] at (145) {\textbf{--}};
	\filldraw (145) circle (1.5pt);
	\fill[orange, fill opacity=0.2] (1) -- (5) -- (4);
\end{tikzpicture}
\begin{tikzpicture}[scale=0.8]
	\coordinate (2) at (0,0);
	\coordinate (4) at (0,3);
	\coordinate (3) at (3,0);
	\coordinate (1) at (-3,-1);
	\coordinate (5) at (2,-2);
	\coordinate (12) at (-1.5,-0.5);
	\coordinate (14) at (-1.5,1);
	\coordinate (15) at (-0.5,-1.5);
	\coordinate (13) at (0,-0.5);
	\coordinate (23) at (1.5,0);
	\coordinate (24) at (0,1.5);
	\coordinate (25) at (1,-1);
	\coordinate (34) at (1.5,1.5);
	\coordinate (35) at (2.5,-1);
	\coordinate (45) at (1,0.5);
	\coordinate (124) at (-1,2/3);
	\coordinate (123) at (0,-1/3);
	\coordinate (135) at (4/3,-1);
	\coordinate (234) at (1,1);
	\coordinate (235) at (5/3,-2/3);
	\coordinate (345) at (5/3,1/3);
	\filldraw (2) circle (1pt);
	\filldraw (3) circle (1pt);
	\filldraw (4) circle (1pt);
	\filldraw (5) circle (1pt);
	\filldraw (1) circle (1pt);
	\node[left] at (2) {$n_2$};
	\node[right] at (3) {$n_3$};
	\node[above] at (4) {$n_4$};
	\node[below] at (5) {$n_0$};
	\node[left] at (1) {$n_1$};
	\draw[dotted] (2)--(3)--(4)--(2)--(1)--(5)--(3)--(1)--(4)--(5)--(2);
	\draw[thick] (235)--(23)--(123);
	\draw[thick] (23)--(234);
	\node[right] at (23) {\textbf{+}};
	\filldraw (23) circle (1.5pt);
	\fill[yellow, fill opacity=0.2] (2) -- (3) -- (4);
	\fill[orange, fill opacity=0.2] (2) -- (3) -- (1);
	\fill[red, fill opacity=0.2] (2) -- (3) -- (5);
\end{tikzpicture}
\caption{The singular locus in dimension $d=3$. The left side (in $A$) depicts a negative vertex, i.e.\ three edges in the singular set meeting at an interior point of a 2-cell. The right side (in $B$) depicts a positive vertex, i.e.\ three edges in the singular set meeting at the midpoint of a 1-cell.}\label{fig:signs}
\end{figure}
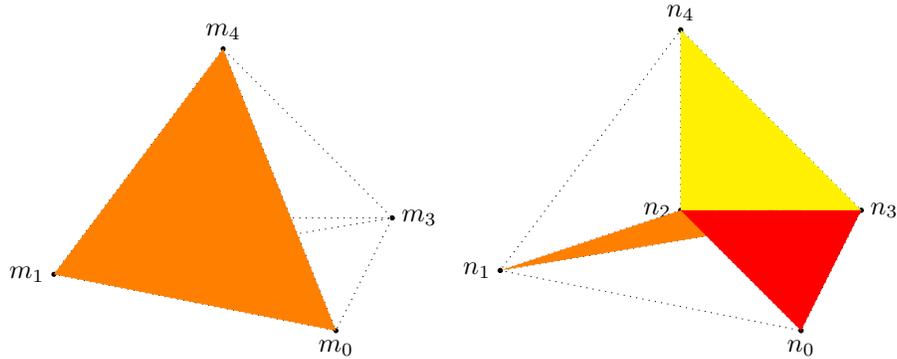

\medskip
Now assume that the densities of the measures in Theorem~B are smooth. As explained above, 
the metrics on $\Lambda_A$ and $\Lambda_B$ defined by $\phi$ and $\psi$ over $A_{\reg}$ and $B_{\reg}$ are smooth and strictly convex, and hence define Riemannian metrics $g_\phi$ and $g_\psi$ there. An easy calculation (see~\cref{lem:isom}) shows that $\partial^c\phi$ and $\partial^c\psi$ induce isometries between $(A_{\reg},g_\phi)$ and $(B_{\reg},g_\psi)$. 

We want to study the metric completions $\overline{(A_{\reg},d_\phi)}$ and $\overline{(B_{\reg},d_\psi)}$ of $(A_{\reg},d_\phi)$ and $(B_{\reg},d_\psi)$, where $d_\phi$ and $d_\psi$ are the distances induced by the Riemannian metrics $g_\phi$ and $g_\psi$. The second completion should, according to a conjecture by Kontsevich and Soibelman~\cite{KS06,GW00}, be the Gromov--Hausdorff limit of certain degenerating families of Calabi--Yau hypersurfaces of complex projective space, as studied in~\cite{HJMM24}. We expect that the metric completions are homeomorphic to $A$ and $B$ equipped with the Euclidean metric. We are not quite able to prove this in general, but we have the following result.
\begin{thmC}\label{thm:MetricCompletion}
Let $\mu$ and $\nu$ be symmetric probability measures on $A$ and $B$ that are absolutely continuous with respect to Lebesgue measure, with densities that are smooth and bounded away from zero and infinity (or, more generally, the densities are smooth and doubling). Assume that $\phi\in\cP_{\sym}$ and $\psi=\phi^c\in\cQ_{\sym}$ satisfy~($\star$). 
Then the identity maps 
\[
(A_{\reg},d_{\mathrm{Eucl}})\to (A_{\reg},d_\phi)
\quad\text{and}\quad
(B_{\reg},d_{\mathrm{Eucl}})\to(B_{\reg},d_\psi)
\]
are H\"older continuous and extend uniquely to H\"older continuous surjective maps 
\[
\chi_A\colon (A,d_{\mathrm{Eucl}})\to\overline{(A_{\reg},d_\phi)}
\quad\text{and}\quad
\chi_B\colon (B,d_{\mathrm{Eucl}})\to\overline{(B_{\reg},d_\psi)}
\]
with connected fibers, satisfying $\chi_A^{-1}(A_{\reg})=A_{\reg}$ and $\chi_B^{-1}(B_{\reg})=B_{\reg}$.
When $d\le2$, these $\chi_A$ and $\chi_B$ are injective, and hence homeomorphisms.
\end{thmC}
The crucial point in Theorem~C is that the maps extend to $A$ and $B$. Concretely, this means that if two Cauchy sequences in $A_{\reg}$ are close with respect to the Euclidean distance, then they are close with respect to $d_\phi$. This follows readily from the $C^{1,\alpha}$-estimate in Theorem~A. In order to prove that the maps are homeomorphisms, one would also need the opposite statement, i.e.\ that any two Cauchy sequences in $A_{\reg}$ that are close with respect to $d_\phi$ are close with respect to the Euclidean distance. However, it seems that this does not follow from the regularity result in Theorem~A (see Appendix~\ref{app:ShortCurves} for a related example). It would follow from a global $C^{1,1}$-estimate on $A_{\reg}$, but in light of \cite{JMPS23}, such an estimate might be too much to hope for.

\bigskip
The paper is organized as follows. In~\cref{sec:Setup}, we recall the setup of~\cite{HJMM24}, and also state the Caffarelli (interior and boundary) regularity in a form that we will use. Then, in~\cref{sec:transport}, we prove that the transport problem~($\star$) admits a solution for suitably regular measures, and prove Theorem~A. Sections~\ref{sec:regonsing} and~\ref{sec:completions} contain the proofs of Theorems~B and~C, respectively, and in Appendix~\ref{app:ShortCurves} we make a calculation illustrating a problem of improving Theorem~C in dimension 3.

\subsection*{Acknowledgement}
We thank S\'ebastien Boucksom and Yang Li for fruitful discussions of the material in and surrounding this note. The third author was supported by NSF grant DMS-2154380 and the Simons Foundation. The fourth author was supported by the collaborative research center SFB 1085 \emph{Higher Invariants - Interactions between Arithmetic Geometry and Global Analysis} funded by the Deutsche Forschungsgemeinschaft.

\section{Setup}\label{sec:Setup}
 
We briefly recall the setup and notation in \cite{HJMM24}. 
Fix an integer $d\geq 1$.  Let $M$ be the lattice $\{y\in \Z^{d+2}: \sum_{i=0}^{d+1}y_i = 0\}$ and $N=\Z^{d+2}/\Z(1,\dots,1)$ its dual. Denote by $e_0, e_1,\dots,e_{d+1}$ the elements $(1,0,\dots,0), (0,1,0,\dots,0),\dots,(0,\dots,0,1)$. From now on, let $\Delta\subset M_\R$ be the convex hull of 
\[
m_i:=(d+1)e_i-\sum_{j\neq i} e_j\in M
\]
for $i=0,\dots,d+1$. 
The polar simplex $\Delta^\vee\subset N_\R$ has vertices $n_0=(-1,\dots,0),\dots,n_{d+1}=(0,\dots,-1)$. We may identify the general simplices in the introduction with $\Delta$ and $\Delta^\vee$ using a suitable identification of $V$ and $M_\R$. Denote by $\sigma_i\subset A:=\partial\Delta$ and $\tau_i\subset B:=\partial\Delta^\vee$ the facets that are dual to the vertices $n_i$ and $m_i$, for $i=0,\dots,d+1$, respectively. 

A useful observation is that any point $m\in A$ is of the form $m=\sum_{i=0}^{d+1} \alpha_k m_k$ for real numbers $\alpha_k$. Furthermore, these numbers are uniquely determined when required to satisfy $\sum_{k=0}^{d+1} \alpha_k=1$ and $\min \alpha_k = 0$. 
Given $m$ in $A$ we will use $\alpha_0(m),\ldots, \alpha_{d+1}(m)$ to denote these numbers. This allows us to define the following subsets $S_i$ of $A$, which we will refer to as the \emph{small stars}. Let 
\[
S_i=\{m\in A:\alpha_i(m)\geq \alpha_j(m)\forall j\}
\]
for $i=0,\dots,d+1$. These sets are closed neighborhoods of the corresponding vertex $m_i$, and subsets of the corresponding closed star $\mathrm{Star}(m_i):=\bigcup_{j\neq i}\sigma_j$ of $m_i$ consisting of all facets of $\Delta$ containing $m_i$. Denote by $T_i$ the similarly defined subsets of $B$.

On $A=\partial \Delta$, and in a similar manner for $B=\partial \Delta^\vee$, we define a singular special affine structure in the following way. First define maps $M_\R\rightarrow \R^d$ for $i\neq j$ given by 
\[
m\mapsto (\langle m, n_j-n_k\rangle)_{k\neq i,j}.
\]
These maps restrict to bijections on $\Star(m_i)$ that we will refer to as $p_{i,j}^{-1}$. 
The atlas making up the singular special affine structure consists of two types of charts, namely $(S_i^\circ,p_{i,j}^{-1})$ for $j\ne i$, and 
$(\sigma_i^\circ,p_{j,i}^{-1})$ for $j\neq i$.
As mentioned in the introduction, this atlas covers a set that we will refer to as $A_{\reg}$. On $B$, there is an analogous construction of an atlas covering a set $B_{\reg}$. 
We will refer to the corresponding charts as $(\tau_i^\circ,q_{j,i}^{-1})$ and $(T_i^\circ,q_{i,j}^{-1})$, respectively. 
For the fact that these make up special affine atlases, i.e.\ that the induced transition functions on $\R^d$ are affine and volume preserving, see~\cite[\S2.1]{HJMM24}. 

As in the introduction, we have the classes of symmetric c-convex functions $\cP_{\sym}\subset C^0(A)$ and  $\cQ_{\sym}\subset C^0(B)$; see~\cite[\S3]{HJMM24}. The c-transform defines isometries between $\cP_{\sym}$ and $\cQ_{\sym}$, and becomes the Legendre transform in the charts above. More precisely, consider $\phi\in\cP_{\sym}$ and $\psi\in\cQ_{\sym}$ with $\psi=\phi^c$. For any index $i$, the c-subgradient $\partial^c\phi$ maps $S_i^\circ$ into $\tau_i$ and $\sigma_i^\circ$ into $T_i$. Moreover, for any $j\ne i$, the function $(\psi - m_j)\circ q_{i,j}$ on $q_{i,j}^{-1}(T_i^\circ)$ is convex, with subgradient image contained in $p_{j,i}^{-1}(\sigma_i)$, and its Legendre transform is the convex function $\phi\circ p_{j,i}$ on $p_{j,i}^{-1}(\sigma_i)$. Similarly, $(\phi-n_j)\circ p_{i,j}$ is convex on $p_{i,j}^{-1}(T_i^\circ)$, with Legendre transform given by the convex function $\psi\circ q_{j,i}$ on $q_{j,i}^{-1}(\tau_i)$. 

As explained in~\cite[\S3.4]{HJMM24}, there is a natural principal $\R$-bundle $\Lambda_A$ on $A$ whose restriction to $A_{\reg}$ is an affine $\R$-bundle. Any function $\phi\in\cP_{\sym}$ induces a continuous metric on $\Lambda_A$ that is convex over on $A_{\reg}$, and we can talk about this metric being smooth or strictly convex. This terminology is convenient, and simply means that the functions $(\phi-n_j)\circ p_{i,j}$ and $\phi\circ p_{j,i}$ are smooth and strictly convex on $S_i^\circ$ and $\sigma_i^\circ$, respectively, for any $i\ne j$. We similarly have a principal $\R$-bundle $\Lambda_B$ on $B$, with analogous properties.

We will say that an absolutely continuous measure $\mu$ on $A$ is doubling if its restriction to open faces and small stars is doubling. More precisely, there should be a constant $C$ such that if $\Omega$ is a convex subset of $p_{i,j}^{-1}(\sigma_j^\circ)\subset \R^d, i\not=j$, 
with center of mass $x$, then 
\[
\mu(\Omega) \leq C\mu\left(x+(\Omega-x)/2\right),
\]
where $x+(\Omega-x)/2$ is the dilation of $\Omega$ by $1/2$ with respect to its center of mass, and similarly for convex subset of $p_{i,j}^{-1}(S_i^\circ)\subset \R^d, i\not=j$.

We will use the same terminology for absolutely continuous measures on $B$.

\medskip
Our analysis relies in a crucial way on the following result by Caffarelli, generalized to doubling measures by Jhaveri and Savin. 
\begin{thm}\label{thm:Caf}
    Let $U$ and $V$ be bounded convex domains in $\R^d$, and $\mu$, $\nu$ probability measures on $U$, $V$ with densities (with respect to Lebesgue measure) bounded away from zero and infinity (or more generally, densities that are doubling). Then there exist differentiable, strictly convex functions $u\colon U\to\R$ and $v\colon V\to\R$, such that $u$ is the Legendre transform of $v$, and 
    \begin{equation}\label{equ:Brenier}
        (\partial u)_*\mu=\nu
        \quad\text{and}\quad
        (\partial v)_*\nu=\mu.
    \end{equation}
    Moreover, $u$ and $v$ are uniformly differentiable and strictly convex in the following sense. There exist constants $\alpha\in(0,1)$, $\beta>0$, and $C>1$ such that 
    \begin{align}
    C^{-1}|x'-x|^{1+\beta}&\le u(x')-u(x)-\partial u(x)(x'-x)\le C|x'-x|^{1+\alpha}\label{equ:Caf1}\\
    C^{-1}|y'-y|^{1+\beta}&\le v(y')-v(y)-\partial v(y)(y'-y)\le C|y'-y|^{1+\alpha}\label{equ:Caf2}
    \end{align}
    for $x,x'\in U$, $y,y'\in V$.
    The pair $(u,v)$ is unique up to replacement by $(u-c,v+c)$ for a constant $c\in\R$. 
    Moreover, if the densities of $\mu$ and $\nu$ are smooth, so are $u$ and $v$.
\end{thm}
Note that the existence and uniqueness (up to a constant) of a convex function $u$ such that~\eqref{equ:Brenier} holds (with $v$ the Legendre transform of $u$) is due to Brenier~\cite{Bre91} and true for more general domains $U$ and $V$. 
The regularity properties of $u$ and $v$ when $U$ and $V$ are convex follow from a combination of the main results in~\cite{Caf92a,Caf92b}, see~\cite[\S1]{AC16} for a discussion and~\cite{JS22} for the doubling case. It follows formally that $u$ and $v$ extend uniquely as continuous convex functions to the closures of $U$ and $V$, and that~\eqref{equ:Caf1},~\eqref{equ:Caf2} continue to hold. Theorem~\ref{thm:Caf} is therefore referred to as \emph{boundary $C^{1,\alpha}$ regularity}.
%
%
%
%
\section{An optimal transport problem}\label{sec:transport}
In this section, we prove Theorem~A. As indicated in the introduction, we in fact consider a more general situation involving optimal transport between two measures.

The c-gradient of any $\phi\in\cP_{\sym}$ defines a  Lebesgue measurable map from $A$ to $B$, so if $\mu\in\cM_{\sym}(A)$ is absolutely continuous with respect to Lebesgue measure $\mu_A$, written $\mu\ll\mu_A$, then $(\partial^c\phi)_*\mu$ is a well-defined symmetric probability measure on $B$. Similarly, if $\nu\in\cM_{\sym}(B)$ satisfies $\nu\ll\mu_B$, then $(\partial^c\psi)_*\nu$ is a well-defined symmetric probability measure on $A$ for any $\psi\in\cQ_{\sym}$.

\begin{thm}\label{thm:transport}
    Assume $\mu\in\cM_{\sym}(A)$ and $\nu\in\cM_{\sym}(B)$ are equivalent to Lebesgue measure $\mu_A$ and $\mu_B$, respectively. Then there exists $\phi\in\cP_{\sym}$, unique up to an additive constant, such that if we set $\psi=\phi^c\in\cQ_{\sym}$, then
    \begin{equation}\label{equ:transports}
    (\partial^c\phi)_*\mu=\nu
    \quad\text{and}\quad
     (\partial^c\psi)_*\nu=\mu
    \end{equation}
\end{thm}
This is a version of~\cite[Theorem~5.2]{HJMM24}, where $\mu=\mu_A$ is Lebesgue measure on $A$, $\nu\in\cM_{\sym}(B)$ is an arbitrary measure, and only the first equation in~\eqref{equ:transports} is considered.
\begin{proof}
Define a functional $F\colon\cP_{\sym}\to\R$ by 
    \[
    F(\phi)=\int_A\phi\,d\mu+\int_B\phi^c\,d\nu.
    \]
    Arguing as in the proof of~\cite[Theorem~5.2]{HJMM24}, one shows that $F$ admits a minimizer, unique up to an additive constant. Further, the reasoning used in the proof of~Lemma~5.1 in loc.~cit.\ shows that $\phi$ is a minimizer for $F$ iff $\phi$ and $\psi:=\phi^c$ satisfy~\eqref{equ:transports}; these proofs only use that $\mu$ and $\nu$ are equivalent to Lebesgue measure.
\end{proof}

We are now ready to state a generalization of Theorem~A in the introduction. 
\begin{thm}\label{thm:breg}
    Assume that $\mu\in\cM_{\sym}(A)$ and $\nu\in\cM_{\sym}(B)$ are equivalent to Lebesgue measure, with densities bounded away from zero and infinity (or more generally, densities that are doubling). Let $\phi\in\cP_{\sym}$ be such that $\phi$ and $\psi:=\phi^c\in\cQ_{\sym}$ satisfy~\eqref{equ:transports}. Then:
    \begin{itemize}
        \item[(i)] 
            the metric on $\Lambda_A$ defined by $\phi$ is uniformly strictly convex, and uniformly $C^{1,\alpha}$ on $A_{\reg}$, in the sense that, for any  distinct indices $i$ and $j$, the functions $(\phi-n_j)\circ p_{i,j}$ on $p_{i,j}^{-1}(S_i^\circ)\subset\R^d$ and $\phi\circ p_{j,i}$ on $p_{j,i}^{-1}(\sigma_i^\circ)\subset\R^d$ satisfy inequalities as in~\eqref{equ:Caf1} and~\eqref{equ:Caf2}; similarly, the metric on $\Lambda_B$ defined by $\psi$ is uniformly strictly convex, and uniformly $C^{1,\alpha}$ on $B_{\reg}$;
        \item[(ii)]
            if the densities of $\mu$ and $\nu$ are $C^\infty$ on $A_{\reg}$ and $B_{\reg}$, respectively, then the metrics on $\Lambda_A$ and $\Lambda_B$ defined by $\phi$ and $\psi$ are $C^\infty$ on $A_{\reg}$ and $B_{\reg}$, respectively,  in the sense that all the functions in~(i) are $C^\infty$.
    \end{itemize}
\end{thm}
\begin{proof}
Pick any two distinct indices $i$ and $j$. The sets $U:=p_{i,j}^{-1}(S_i^\circ)$ and $V:=q_{j,i}^{-1}(\tau_i^\circ)$ are bounded convex subsets of $\R^d$. Further, the measures $\tilde\mu:=(p_{i,j}^{-1})_*(\mu|_{S_i^\circ})$ and $\tilde\nu:=(q_{j,i}^{-1})_*(\nu|_{\tau_i^\circ})$ on $U$ and $V$ have densities (with respect to Lebesgue measure) bounded away from 0 and infinity. Finally, the convex functions $g:=(\phi-n_j)\circ p_{i,j}$ on $U$ and $\psi\circ q_{j,i}$ on $V$ are Legendre conjugate, and satisfy $(\partial g)_*\tilde\mu=\tilde\mu$, $(\partial h)_*\tilde\nu=\tilde\nu$, as follows from~\eqref{equ:transports} and~\cite[Lemma~4.4]{HJMM24}.

The statements in~(i)--(ii) now follows from Caffarelli's boundary regularity as in Theorem~\ref{thm:Caf}.
\end{proof}
%
%
%
%
\section{Regularity on the singular set: proof of Theorem~B}\label{sec:regonsing}
The idea of the proof of Theorem~B is to use Caffarelli's boundary regularity result in the form of Theorem~\ref{thm:breg}. This gives strict convexity and differentiability in charts, up to the boundary. To fully treat the c-transform on the singular set, we need a few more arguments involving the symmetry. 
To this end, we introduce the following notation.
Given disjoint (but possibly empty) subsets $I,J\subset\{0,1,\dots,d+1\}$, we set
\[
A_{IJ}:=\bigcap_{i\in I}S_i\cap\bigcap_{j\in J}\sigma_j
\qand
B_{JI}:=\bigcap_{j\in J}T_j\cap\bigcap_{i\in I}\tau_i,
\]
These are compact subsets of $A$ and $B$, respectively. 
Additionally, set
\[
A_{IJ}^\circ:=A_{IJ}\setminus\left(\bigcup_{l\not\in I}S_l\cup\bigcup_{l\not\in J}\sigma_l\right)
\qand
B_{JI}^\circ:=B_{JI}\setminus\left(\bigcup_{l\not\in J}T_l\cup\bigcup_{l\not\in I}\tau_l\right).
\]
The locally closed sets $A_{IJ}^\circ$ and $B_{JI}^\circ$, which are empty whenever $I$ or $J$ is empty, form partitions of $A$ and $B$. 
The points in $A_{IJ}^\circ$ are of the form $\sum_l\alpha_lm_l$, where $\sum_l\alpha_l=1$, $\min_l\alpha_l=0$, $\alpha_l<\alpha_i=\alpha_{i'}$ for $i,i'\in I$, $l\not\in I$, and $0=\alpha_j<\alpha_l$ for $j\in J$, $l\not\in J$. We can similarly describe the points in $B_{JI}^\circ$.

Note that $A_{\sing}$ is the union of all $A_{IJ}$ with $|I|,|J|\ge2$, and similarly for $B_{\sing}$.
The stabilizer of any point in $A_{IJ}^\circ$ or $B_{JI}^\circ$ is the product of the groups $G_I$ and $G_J$ of permutations of $I$ and $J$.

We first prove a lemma that is valid for arbitrary symmetric c-convex functions.
\begin{lem}\label{lem:GradientInclusion}
    Pick any $\phi\in\cP_{\sym}$.
    \begin{itemize}
        \item[(i)]
            For any disjoint nonempty subsets $I$, $J$ of $\{0,1,\dots,d+1\}$ we have 
            \[(\partial^c\phi)(A_{IJ}^\circ)\subset\bigcup_{j\in J}T_j\cap\bigcup_{i\in I}\tau_i.
            \]
        \item[(ii)] 
            For any $i\in I$ we have:
            \begin{itemize}
                \item[(a)]
                    $(\partial^c\phi)(m)\cap\tau_i\ne\emptyset$ when $m\in S_i$;
                \item[(b)]  
                $(\partial^c\phi)(m)\cap T_i\ne\emptyset$ when $m\in\sigma_i$.
            \end{itemize}
    \end{itemize}
    The analogous results hold for any $\psi\in\cQ_{\sym}$.
\end{lem}
Part~(i) generalizes~\cite[Lemma~4.1]{HJMM24}, which says that $(\partial^c\phi)(\sigma_i^\circ)\subset T_i$ and $(\partial^c\phi)(S_i^\circ)\subset\tau_i$.
\begin{proof}
    Set $\psi:=\phi^c\in\cQ_{\sym}$.
    To prove~(i), consider $n\in(\partial^c\phi)(m)$ and write $m=\sum_l\alpha_lm_l$, $n=\sum_l\beta_ln_l$, where $\sum\alpha_l=\sum\beta_l=1$ and $\min_l\alpha_l=\min_l\beta_l=0$. 
    Suppose $n\not\in\bigcup_{j\in J}T_j$. Then there exists $l\not\in J$ such that $\beta_l>\beta_j$ for all $j\in J$. Pick any $j\in J$, and let $g\in G$ be the transposition exchanging $j$ and $l$. Now $\alpha_l>\alpha_j$, and a simple calculation as in the proof of~\cite[Lemma~2.1]{HJMM24} shows that $\langle m,g(n)\rangle-\langle m,n\rangle=(d+2)(\alpha_l-\alpha_j)(\beta_l-\beta_j)>0$, and hence $\langle m,g(n)\rangle-\psi(g(n))>\langle m,n\rangle-\psi(n)$ by $G$-invariance of $\psi$, contradicting $n\in(\partial^c\phi)(m)$. A similar argument shows that $n\in\bigcup_{i\in I}\tau_i$. 

    As for~(ii), pick $m\in S_i$, and a sequence $(m_r)_r$ in $S_i^\circ$ converging to $m$. Then $n_r:=(\partial^c\phi)(m_r)\in\tau_i$, so by compactness, we may assume $n_r$ converges to some element $n\in\tau_i$. Now, for any $r$, we have 
    \[
    \phi(m_r)+\psi(n_r)=\langle m_r,n_r\rangle,
    \]
    and continuity gives
    \[
    \phi(m)+\psi(n)=\langle m,n\rangle,
    \]
    showing that $n\in(\partial^c\phi)(m)$. 

    A similar argument shows that $(\partial^c\phi)(\sigma_i)\cap T_i\ne\emptyset$ and completes the proof.
\end{proof}
\begin{proof}[Proof of Theorem~B]
    We proceed in several steps, exploiting the symmetry together with the boundary regularity.

    \smallskip
    \textbf{Step 1}: We have $(\partial^c\phi)(A_{IJ}^\circ)=B_{JI}^\circ$, $(\partial^c\phi)(A_{IJ})=B_{JI}$, $(\partial^c\psi)(B_{JI}^\circ)=A_{IJ}^\circ$, and $(\partial^c\psi)(B_{JI})=A_{IJ}$ for any nonempty disjoint subsets $I$ and $J$.

    \smallskip
    It suffices to prove the statement for $\phi$, and it also suffices to prove $(\partial^c\phi)(A_{IJ}^\circ)=B_{JI}^\circ$, since $A_{IJ}$ (resp. $B_{JI}$) is the disjoint union of $A_{I'J'}^\circ$ (resp.\ $B_{J'I'}$) over all disjoint sets $I'$, $J'$ containing $I$, $J$.
    
    We first prove that $(\partial^c\phi)(A_{IJ}^\circ)\subset B_{JI}$. Pick any $m\in A_{IJ}^\circ$, and $n\in(\partial^c\phi)(m)$.
    By Lemma~\ref{lem:GradientInclusion}, there exist $i\in I$ and $j\in J$ such that $n\in T_j\cap\tau_i$.
    First suppose $n\not\in\bigcap_{l\in J}T_l$. Then there exists a permutation $g$ of $J$ such that $g(n)\ne n$. Note that $g(m)=m$ since $m\in A_{IJ}$. Consequently, $g(n)\in \partial^c\phi(m)$. Set $n_\theta:=(1-\theta)n+\theta g(n)$ for $0\le\theta\le 1$. Then $\theta\mapsto n_\theta$ is a nonconstant affine path in $\tau_i$. By Theorem~\ref{thm:breg}~(ii), $\theta\mapsto\psi(n_\theta)$ is strictly convex, whereas the function $\theta\mapsto\langle m,n_\theta\rangle$ is affine. It follows that the function
    \[
    \chi(\theta):=\psi(n_\theta)+\phi(m)-\langle m,n_\theta\rangle\
    \]
    on $[0,1]$ is strictly convex. Now $\chi\ge 0$, whereas $\chi(0)=\chi(1)=0$, a contradiction. Thus $n\in T_J$.
    The case when $n\not\in\bigcap_{l\in I}\tau_l$ similarly leads to a contradiction, using Theorem~\ref{thm:breg}~(i).

    Thus $(\partial^c\phi)(A_{IJ}^\circ)\subset B_{JI}$. We claim that in fact $(\partial^c\phi)(A_{IJ}^\circ)\subset B_{JI}^\circ$. This follows from the $G$-equivariance of $\partial^c\phi$, and the fact that a point in $A_{IJ}$ (resp.\ $B_{JI}$) belongs to $A_{IJ}^\circ$ (resp.\ $B_{JI}^\circ$) iff its stabilizer is exactly equal to the product of the subgroups $G_I$ and $G_J$ of permutations of $I$ and $J$.

    Thus $(\partial^c\phi)(A_{IJ}^\circ)\subset B_{JI}^\circ$. In fact, equality must hold since $\partial^c\phi$ is a surjective (muti-valued) map, and since $\{A_{IJ}^\circ\}_{I,J}$ and $\{B_{JI}^\circ\}_{I,J}$ form partitions of $A$ and $B$, respectively.
    
    \smallskip
    \textbf{Step 2}: the c-subgradient maps $\partial^c\phi\colon A\to B$ and $\partial^c\psi\colon B\to A$ are everywhere single-valued. 
    
    \smallskip 
    Indeed, pick any $m\in A$. There exist unique (nonempty) $I$ and $J$ such that $m\in A_{IJ}^\circ$. By Step~1, we have $(\partial^c\phi)(m)\subset B_{JI}^\circ$. Now suppose $(\partial^c\phi)(m)$ contains two distinct points $n,n'$. If $i\in I$, then $n,n'\in\tau_i$. As in Step~1, we get a contradiction by considering the function 
    \[
    [0,1]\ni\theta\mapsto\psi((1-\theta)n+\theta n')-(1-\theta)\langle m,n\rangle-\theta\langle m,n'\rangle,
    \]
    which is strictly convex and attains its minimum at $\theta=0$ and $\theta=1$.    

    \smallskip
    As a consequence, the c-gradients $\partial^c\phi$ and $\partial^c\psi$ are bijections, and inverse one to another.

    \smallskip
    \textbf{Step 3}:
    the c-gradient maps $\partial^c\phi\colon A\to B$ and 
    $\partial^c\psi\colon B\to A$ are H\"older continuous homeomorphisms.

    \smallskip
    As $A$, $B$ are compact Hausdorff spaces, it suffices to prove that $\partial^c\phi$ and $\partial^c\psi$ are (H\"older) continuous, and by symmetry we only need to show that  $\partial^c\phi$ is H\"older continuous on $S_i$ for each $i$. Pick any $j\ne i$ and set $\phi_{i,j}:=(\phi-n_j)\circ p_{i,j}$. On the one hand, it follows from~\cref{thm:breg} that $\phi_{i,j}$ is differentiable and strictly convex, and that $\partial\phi_{i,j}$ extends to a H\"older continuous homeomorphism of $p_{i,j}^{-1}(S_i)$ onto $q_{j,i}^{-1}(\tau_i)$. On the other hand, it follows from~\cite[Lemma~4.4]{HJMM24} that
    \[
    \phi(p_{i,j}(m))+\psi(q_{j,i}(\partial\phi_{i,j}(p_{i,j}(m)))=\langle p_{i,j}(m),q_{i,j}((\partial\phi_{i,j})(m))\rangle
    \]
    for all $m\in S_i^\circ$. This equality therefore also holds for $m\in S_i$, and we have $\partial^c\phi=q_{j,i}\circ(\partial\phi_{i,j})\circ p_{i,j}^{-1}$ on $S_i$, which shows that $\partial^c\phi$ is continuous on $S_i$.
\end{proof}
\begin{rmk}\label{rmk:3dvert1}
    In dimension $d=3$, the singular set $A_{\sing}$ is a trivalent graph, consisting of 30 edges of the form $A_{I,J}$, where $|I|=|J|=2$, 10 ``positive" vertices, of the form $A_{IJ}$, with $|I|=2$, $|J|=3$, and 10 ``negative" vertices, of the form $A_{IJ}$, with $|I|=3$, $|J|=2$. Each edge of $A_{\sing}$ contains one positive vertex and one negative vertex. The singular set $B_{\sing}$ has the same structure. Now the c-gradient maps $\partial^c\phi$ and $\partial^c\psi$ restrict to homeomorphisms between the edges $A_{IJ}$ and $B_{JI}$, mapping a positive vertex to a negative vertex, and conversely. See Figure~\ref{fig:signs}.
\end{rmk}
%
%
%
%
\section{Metric completions: proof of Theorem~C}\label{sec:completions}
Let us recall the setting. We have probability measures $\mu\in\cM_{\sym}(A)$ and $\nu\in\cM_{\sym}(B)$ that are absolutely continuous with respect to Lebesgue measure, with smooth densities bounded away from zero and infinity (or more generally, with doubling), and $\phi\in\cP_{\sym}$ satisfies $(\partial^c\phi)_*\mu=\nu$, $(\partial^c\psi)_*\nu=\mu$, where $\psi=\phi^c$.

By Theorem~\ref{thm:breg}, the metric on $\Lambda_A$  associated to $\phi$ is smooth and strictly convex on $A_{\reg}$, and gives rise to a Hessian metric $g_\phi$ there. Similarly, we obtain a Hessian metric $g_\psi$ on $B_{\reg}$. 

\begin{lem}\label{lem:isom}
The c-gradient map $\partial^c\phi\colon(A_{\reg},g_\phi)\to(B_{\reg},g_\psi)$ is an isometry.
\end{lem}
\begin{proof}
    By passing to coordinates, and using \cite[Lemma 4.4]{HJMM24}, this follows from a corresponding statement about the Legendre transform on $\R^d$. Consider a smooth, strictly convex function $f(x)$ with Legendre transform $g(y)$. The relation $g(\nabla f(x))=x\cdot \nabla f(x) -f(x)$ implies $\nabla g \circ \nabla f(x) =x$ and $H(g)\circ\nabla f(x) H(f)(x)=\mathrm{Id}$ where $H(\cdot)$ is the Hessian. Pulling back the metric $\sum _{i,j}H(g)_{i,j}\mathrm{d}y_i\otimes \mathrm{d}y_j$ along $y=\nabla f(x)$ gives
    \begin{align*}
    (\nabla f)^*(\sum_{i,j}H(g)_{i,j}\mathrm{d}y_i\otimes \mathrm{d}y_j)(x)&=\sum_{i,j}(H(g)\circ \nabla f(x) (H(f)(x))^2)_{i,j}\mathrm{d}x_i\otimes \mathrm{d}x_j\\
    &=\sum_{i,j}H(f)(x)_{i,j}\mathrm{d}x_i\otimes\mathrm{d}x_j,
    \end{align*}
    and we are done.
\end{proof}

Let us write $d_\phi$ and $d_\psi$ for the metrics (distances) on $A_{\reg}$ and $B_{\reg}$ induced by $g_\phi$ and $g_\psi$. We also write $d_{\mathrm{Eucl}}$ for the metric on $A$ induced by any Euclidean metric on $M_\R$ (via an isomorphism $M\simeq\Z^{d+1}$). The Euclidean metric on $B$ is also denoted $d_{\mathrm{Eucl}}$.

The main result needed for the proof of Theorem~C is the following estimate.
\begin{lem}\label{lem:DistUpperBound}
    There exists $C>0$ and $\beta\in(0,1)$ such that
    \begin{equation}\label{equ:dphimaj}
    d_\phi(x,y)\le Cd_{\mathrm{Eucl}}(x,y)^\beta
    \end{equation}
    for all $x,y\in A_{\reg}$. Moreover, for any compact subset $K$ of $A_{\reg}$ there exists $\delta>0$ such that
    \begin{equation}\label{equ:dphimin}
    d_\phi(x,y)\ge\delta d_{\mathrm{Eucl}}(x,y)
    \end{equation}
    for $x,y\in K$. The analogous estimates also hold for the distances $d_\psi$ and $d_{\mathrm{Eucl}}$ on $B_{\reg}$.
\end{lem}
\begin{proof}
    To prove~\eqref{equ:dphimaj}, first suppose that $x,y\in\sigma_i^\circ$ for some $i$.
    Set $\gamma(t)=(1-t)x+ty$ for $0\le t\le 1$. 
    By Theorem~\ref{thm:breg}, the function $h:=\phi\circ\gamma$ is smooth and strictly convex on $[0,1]$, and we have 
      \begin{equation*}
      d_\phi(x,y)
      \le \int_0^1\sqrt{h''(t)}\,dt
      \le \sqrt{\int_0^1h''(t)\,dt}
      =\sqrt{h'(1)-h'(0)}
      \lesssim\sqrt{d_{\mathrm{Eucl}}(x,y)^{1+\alpha}},
    \end{equation*}
    using Jensen's inequality, and Theorem~\ref{thm:breg}.
    A similar argument works if $x,y\in S_j^\circ$ for some $j$.
    In the general case, we can find a sequence of points 
    \[x=x_0,x_1,\dots,x_m=y,\]
    with $m\le 2d$, such that for each $l$ we have $x_l,x_{l-1}\in\sigma_i^\circ$ for some $i$ or $x_l,x_{l-1}\in S_j^\circ$ for some $j$, and such that $d_{\mathrm{Eucl}}(x_l,x_{l-1})\lesssim d_{\mathrm{Eucl}}(x,y)$. We then obtain the desired estimate using the triangle inequality.

    Finally, the lower bound in~\eqref{equ:dphimin} follows from the fact that the metric on $\Lambda_A$ defined by $\phi$ is smooth and strictly convex on $A_{\reg}$.
\end{proof}
\begin{proof}[Proof of Theorem~C]
We prove the properties for $A$; the ones for $B$ are proved in the same way.

    It follows from~\eqref{equ:dphimaj} that 
any Cauchy sequence in $(A_{\reg},d_{\mathrm{Eucl}})$ is also a Cauchy sequence in $(A_{\reg},d_\phi)$; hence the identity map extends uniquely to a H\"older continuous map $\chi_A\colon(A,d_{\mathrm{Eucl}})\to\overline{(A_{\reg},d_\phi)}$. Similarly,~\eqref{equ:dphimin} implies $\chi_A^{-1}(A_{\reg})=A_{\reg}$.

To see that $\chi_A$ is surjective, pick any $x\in\overline{(A_{\reg},d_\phi)}$, defined by a Cauchy sequence $(x_i)_i$ in $(A_{\reg},d_\phi)$. As $(A,d_{\mathrm{Eucl}})$ is compact, $(x_i)_i$ has a convergent subsequence $(x_{i_j})_j$, say with limit $x'\in A$, with respect to $d_{\mathrm{Eucl}}$. By~\eqref{equ:dphimaj}, the sequence $(x_{i_j})_j$ maps to a Cauchy sequence in $(A_{\reg},d_\phi)$ equivalent to the given one, so that $\chi_A(x')=x$, proving surjectivity. 

Next we prove that the fibers of $\chi_A$ are connected.
Consider $x,y\in A$ with $x\ne y$ and $\chi_A(x)=\chi_A(y)$. By what precedes we have $x,y\in A_{\sing}$. Thus there exist Cauchy sequences $(x_i)_i$ and $(y_i)_i$ in $A_{\reg}$, converging to $x$ and $y$ respectively with respect to $d_{\mathrm{Eucl}}$ (thus also with respect to $d_\phi$ by~\eqref{equ:dphimaj}), such that $d_\phi(x_i,y_i)\rightarrow 0$. For each $i$, let $\gamma_i$ be a smooth curve in $A_{\reg}$ from $x_i$ to $y_i$ such that the length $\ell_i$ of $\gamma_i$ in $(A_{\reg},g_\phi)$ satisfies $\ell_i\le d_\phi(x_i,y_i)+i^{-1}$. As $A$ is compact, after passing to a subsequence, we may assume that the compact subsets $\gamma_i$ of $(A,d_{\mathrm{Eucl}})$ converge in Hausdorff distance to some compact subset $\Gamma$ containing $x$ and $y$; since the $\gamma_i$ are connected, so is $\Gamma$, see~\cref{lem:Gamma conn}.
    
We claim that $\Gamma$ is contained in the fiber of $x$ and $y$. To see this, pick any $z\in \Gamma$, and let $(z_i)_i$ be a sequence in $A_{\reg}$ such that $z_i \in \gamma_i$ and $d_{\mathrm{Eucl}}(z_i,z) \rightarrow 0$. Let $\epsilon>0$ and assume $i$ is large enough that $d_{\mathrm{Eucl}}(x_i,x) < \epsilon$, $d_{\mathrm{Eucl}}(z_i,z) < \epsilon$, and $\ell_i<\epsilon$. By~\eqref{equ:dphimaj} we get
\begin{align*}
  d_\phi(\chi_A(x),\chi_A(z))
  & \leq d_\phi(\chi_A(x),\chi_A(x_i))+ d_\phi(\chi_A(x_i),\chi_A(z_i)) +d_\phi(\chi_A(z_i),\chi_A(z)) \\
  & \leq  Cd_{\mathrm{Eucl}}(x,x_i)^\beta+ \ell_i + Cd_{\mathrm{Eucl}}(z_i,z)^\beta < \epsilon + 2C\epsilon^\beta.
\end{align*}
As $\epsilon$ is arbitrary small, we get $d_\phi(\chi_A(x),\chi_A(z))=0$, and hence $\chi_A(x)=\chi_A(z)$.

  Finally we prove that $\chi$ is injective in dimension $d\le 2$. When $d=1$, we have $A_{\reg}=A$ and there is nothing to prove. When $d=2$, $A_{\sing}$ is a finite set (consisting of six points), so the fact that the fibers of $\chi_A$ are connected implies that they are singletons.
\end{proof}

\begin{lem} \label{lem:Gamma conn}
    Let $(\Gamma_i)_i$ be a sequence of closed subsets of a compact metric space $(X,d)$ converging in Hausdorff distance to a nonempty subset $\Gamma$. If the $\Gamma_i$ are connected, then so, too, is $\Gamma$.
\end{lem}

\begin{proof}
    We recall from the theory of metric spaces that
    \begin{align} \label{eq:Gamma}
    \begin{split}
        \Gamma
        & = \{ x \in X \,|\, \forall \varepsilon>0 \quad \lvert \{ i \,|\, B(x,\varepsilon) \cap \Gamma_i = \emptyset \} \rvert < \infty \}  \\
        & = \{ x \in X \,|\, \forall \varepsilon>0 \quad  \lvert \{ i \,|\, B(x,\varepsilon) \cap \Gamma_i \neq \emptyset \}\rvert = \infty \}.
    \end{split}
    \end{align}
    Assume by contradiction that $\Gamma$ is not connected, so we can write $\Gamma=V_1\cup V_2$ with $V_1, V_2$ closed disjoint subsets. There exist disjoint open subsets $U_1$ and $U_2$ containing $V_1$ and $V_2$ respectively (as $U_1$ take for instance the union over the points $x \in V_1$ of the open balls $B(x,\tfrac{d(x,V_2)}{3})$). Then $U_1 \cup U_2$ is an open containing $\Gamma$, and for i large enough $\Gamma_i \subset U_1 \cup U_2$. Suppose that, for any $i$ sufficiently large, we have $\Gamma_i \subset U_1$ or $\Gamma_i \subset U_2$; without loss of generality assume $\Gamma_i \subset U_1$ for infinitely many $i$. Then, for $x \in V_2$, the open $U_2$ contains $x$ and doesn't intersect infinitely many $\Gamma_i$, which is impossible by \cref{eq:Gamma}. Thus, there exists $i$ such that $\Gamma_i=(\Gamma_i \cap U_1)\cup ( \Gamma_i \cap U_2)$ and $\Gamma_i \cap U_j \neq \emptyset$ for $j=1,2$. This implies that $\Gamma_i$ is not connected, a contradiction which concludes the proof. 
\end{proof}

\begin{rmk}\label{rmk:3dvert2}
In dimension $d=3$, the singular locus $A_{\sing}$ is a trivalent graph, see Figure~\ref{fig:signs} and~\cref{rmk:3dvert1}. Using Theorem~C, we can write the metric completion of $(A,d_\phi)$ as $A_{\reg}\sqcup A'_{\sing}$, and we have a continuous surjection $\chi\colon A_{\sing}\to A'_{\sing}$ with connected fibers. We expect this to be a homeomorphism, but our methods, based on boundary regularity, do not seem to yield this. In fact, we are not able to rule out that $A'_{\sing}$ is a singleton, even in the case when $\mu$ and $\nu$ are both equal to Lebesgue measure. Indeed, suppose that $m$ and $m'$ are adjacent vertices in $A_{\sing}$. We need to rule out the existence of curves $\gamma_i$ in $A_{\reg}$ that are short in the $g_\phi$-metric, but whose endpoints converge to $m$ and $m'$, respectively. see Appendix~\ref{app:ShortCurves} for a discussion.
\end{rmk}

\appendix
\section{$C^{1,\alpha}$-metrics with short curves}\label{app:ShortCurves}
Consider the situation in Theorem~C in dimension $d=3$. Pick two points $x,y$ on one of the segments in $A_{\sing}$. One could imagine that there exist points $x',y'$ in $\sigma_i^\circ$ such that $d_{\mathrm{Eucl}}(x,x')$ and $d_{\mathrm{Eucl}}(y,y')$ are small, but that $d_\phi(x',y')$ is also small. Then the map from $(A,d_{\mathrm{Eucl}})$ to the metric completion of $(A,d_\phi)$ may identify $x$ and $y$.
The example below illustrates the problem.
\begin{exam}
    Suppose that $\varphi\colon[0,1]\to\R$ is a convex function for which there exist $C>0$, $\alpha\in(0,1)$ and $\beta>1$ such that 
    \[
    C^{-1}(y-x)^\beta\le\varphi'(y)-\varphi'(x)\le C(y-x)^\alpha
    \]
   for all $0<x<y<1$. One may then wonder whether there exists  $\epsilon=\epsilon(C,\alpha,\beta)>0$ such that 
   \[
   \int_0^1\sqrt{\varphi''(t)}\,dt\ge\epsilon.
   \]
   In other words, $\varphi$ defines a Hessian metric on $(0,1)$ and we ask whether the length of $(0,1)$ in this metric is bounded below by a constant that only depends on $C,\alpha,\beta$.

   The following example gives a negative answer.
   Pick coprime integers $M,N\gg1$ such that 
   \begin{equation} \label{eq:MNchoice} M^\alpha\le N^{1-\alpha} \end{equation}
   and define $f\colon[0,1)\to\R_{>0}$ as a piecewise constant function given by 
   \[
   f|_{[\tfrac{j}{MN},\tfrac{j+1}{MN})}=
   \begin{cases}
       M &\text{if $M\vert j$}\\
       N^{1-\beta} &\text{otherwise}
   \end{cases}
   \]
  for $0\le j<MN$. Then define a convex function $\varphi\colon[0,1]\to\R_{\ge0}$ by setting 
  \[
    \varphi(x)=\int_0^x(x-t)f(t)\,dt.
  \]
  This function satisfies $\varphi(0)=\varphi'(0)=0$ and $\varphi''=f$. Moreover, 
  \[
  0\le\int_0^1\sqrt{\varphi''(t)}\,dt
  =\frac1M\sqrt{M}+\frac{M-1}{M}N^{-\frac{\beta-1}{2}}
  \ll1.
  \]
    We claim that 
    \begin{equation}
    \label{eq:BoundOnDerivative}
    \frac{1}{2}(y-x)^\beta\le\varphi'(y)-\varphi'(x)\le 3(y-x)^\alpha
    \end{equation}
   for all $0<x<y<1$. This amounts to proving that 
    \begin{equation}
    \label{eq:BoundOnIntegral}
    \frac{1}{2}|I|^\beta\le\int_If\le 3|I|^\alpha
    \end{equation}    
for any subinterval $I\subset[0,1]$. We will begin with the first inequality of \eqref{eq:BoundOnIntegral}. 
We get two cases:
\begin{itemize}
\item either $|I|\leq 1/N$
\item or $|I|\geq 1/N$ and $|f^{-1}(M)\cap I| \geq \frac{1}{2M}|I|$.
\end{itemize}
In the first case, we have
\[
 \int_I f \geq N^{1-\beta}|I| = \left(\frac{1}{N}\right)^{\beta-1}|I| \geq |I|^{\beta-1}|I| = |I|^\beta.
\]
In the second case, we have 
\[
\int_I f \geq \frac{M}{2M}|I| = |I|/2 \geq |I|^\beta/2
\]
since $|I|\leq 1$. 

We now prove the second inequality of \eqref{eq:BoundOnIntegral}. Similarly as above, we get two cases:
\begin{itemize}
    \item either $\frac{k-1}{MN}\leq |I|\leq \frac{k}{MN}$ for some $k\in \{1, \ldots, M\}$ and $|f^{-1}(M)\cap I| \leq \frac{2}{k}|I|$
    \item or $|I|\geq 1/N$ and $|f^{-1}(M)\cap I| \leq \frac{2}{M}|I|$.
\end{itemize}
In the first case, we have
\begin{eqnarray*} 
\int_I f & \leq & N^{1-\beta}|I| + M\frac{2}{k}|I| 
= \left(N^{1-\beta} + \frac{2M}{k}\right)|I|^{1-\alpha} |I|^\alpha \\
& \leq & \left(N^{1-\beta} + \frac{2M}{k}\right)\frac{k^{1-\alpha}}{M^{1-\alpha}N^{1-\alpha}} |I|^\alpha \\
& \leq & \left(N^{1-\beta} + \frac{2M}{k}\right)\frac{k^{1-\alpha}}{M^{1-\alpha}M^\alpha} |I|^\alpha 
= \left(N^{1-\beta}\frac{k^{1-\alpha}}{M} + 2k^{-\alpha}\right) |I|^\alpha 
\leq 3 |I|^\alpha,
\end{eqnarray*}
where the third inequality uses \eqref{eq:MNchoice}. In the second case, we have
\[
\int_I f \leq N^{1-\beta}|I| + M\frac{2}{M}|I| = \left(N^{1-\beta} +2\right)|I| \leq 3|I| \leq 3|I|^\alpha,
\]
since $|I|\leq 1$.
\end{exam}

%
%
%
%

\end{document}